\documentclass[
11pt,                          
draft,                         
english                        
]{article}

\usepackage[english]{babel}    
\usepackage{amsmath}           
\usepackage[amsmath,thmmarks,hyperref]{ntheorem} 
\usepackage{amssymb}           
\usepackage[utf8]{inputenc}    
\usepackage[T1]{fontenc}       
\usepackage{exscale}           
\usepackage[sort]{cite}        
\usepackage{hyperref}
\usepackage{eucal}             
\usepackage[a4paper]{geometry} 
\usepackage{xspace}            
\usepackage{tikz}              
\usepackage[expansion=false    
           ]{microtype}        
\usepackage[nottoc]{tocbibind} 
\usepackage[all]{xy}
\usepackage{slashed}
\usepackage{cleveref}
\usetikzlibrary{matrix}
\usepackage{tikz-cd}
\usepackage{bbm}               
\usepackage{paralist}          
\usepackage{mathrsfs}          
\usepackage{stmaryrd}          
\usepackage{mathtools}         

\renewcommand{\mathbb}[1]{\mathbbm{#1}}

\newcommand{\refitem}[1] {\textit{\ref{#1}.)}}

\numberwithin{equation}{section}

\allowdisplaybreaks

\let\originalleft\left
\let\originalright\right
\renewcommand{\left}{\mathopen{}\mathclose\bgroup\originalleft}
\renewcommand{\right}{\aftergroup\egroup\originalright}

\theoremheaderfont{\normalfont\bfseries}
\theorembodyfont{\itshape}
\newtheorem{lemma}{Lemma}[section]
\newtheorem{proposition}[lemma]{Proposition}
\newtheorem{theorem}[lemma]{Theorem}

\newtheorem{definition}[lemma]{Definition}

\theorembodyfont{\rmfamily}

\newtheorem{example}[lemma]{Example}
\newtheorem{remark}[lemma]{Remark}

\makeatletter
\def\theorem@checkbold{}
\makeatother

\theoremheaderfont{\scshape}
\theorembodyfont{\normalfont}
\theoremstyle{nonumberplain}
\theoremseparator{:}
\theoremsymbol{\hbox{$\heartsuit$}}
\newtheorem{proof}{Proof}

\newenvironment{lemmalist}{\begin{compactenum}[\itshape i.)]}{\end{compactenum}}

\newcommand{\Hom}            {\operatorname{\mathsf{Hom}}}
\newcommand{\Sec}[1][k]      {\Gamma^{#1}}
\newcommand{\Secinfty}       {\Sec[\infty]}
\newcommand{\group}[1]        {\mathrm{#1}}

\DeclarePairedDelimiter{\Schouten}{\llbracket}{\rrbracket}

\newcommand{\Anti}                    {\Lambda}

\newcommand{\Fun}[1][k]      {\mathscr{C}^{#1}}
\newcommand{\Cinfty}         {\Fun[\infty]}
\newcommand{\tensor}[1][{}]           {\mathbin{\otimes_{\scriptscriptstyle{#1}}}}
\newcommand{\D}              {\mathop{}\!\mathrm{d}}
\newcommand{\argument}       {\,\cdot\,}
\DeclareMathOperator{\id}    {\mathsf{id}}
\newcommand{\pr}             {\mathrm{pr}}

\newcommand{\Rf}  {\mathbb{R}^d_{\mathrm{formal}}}

\newcommand{\Dpoly}[1]	{\mathrm{D}_{\mathrm{poly}}^{#1}} 
\newcommand{\Tpoly}[1]{\mathrm{T}_{\mathrm{poly}}^{#1}}
\newcommand{\Cpoly}[1]{\mathrm{C}^{\mathrm{poly}}_{#1}}

\newcommand{\SM}	{\mathcal{S}M}
\newcommand{\FibT}[1]{\mathcal{T}_{\mathrm{poly}}^{#1}}
\newcommand{\FibD}[1]{\mathcal{D}_{\mathrm{poly}}^{#1}}
\newcommand{\FibC}[1]{\mathcal{C}_{\mathrm{poly}}^{#1}}
\def\presuper#1#2%
{\mathop{}%
   \mathopen{\vphantom{#2}}^{#1}%
   \kern-\scriptspace%
   #2}
\newcommand{\FormsT}[3]{\presuper{}{\Omega^{#1}(#2;\FibT{#3})}}
\newcommand{\FormsD}[3]{\presuper{}{\Omega^{#1}(#2;\FibD{#3})}}
\newcommand{\FormsC}[3]{\presuper{}{\Omega^{#1}(#2;\FibC{#3})}}

\newcommand{\conn}{\nabla}

\newcommand{\R}{\mathbb{R}}

\title{$L_\infty$-resolutions and twisting in the curved context}

\author{
  \textbf{Chiara Esposito}\thanks{\texttt{chiara.esposito@mathematik.uni-wuerzburg.de}},\\
    Institut für Mathematik \\
  Lehrstuhl für Mathematik X \\
  Universität Würzburg \\
  Campus Hubland Nord \\
  Emil-Fischer-Straße 31 \\
  97074 Würzburg \\
  Germany \\
  [0.3cm]
  \textbf{Niek de Kleijn}\thanks{\texttt{niekdekleijn@gmail.com }},\\
  Korteweg-de Vries Institute for Mathematics\\
  University of Amsterdam \\
  Science Park 105-107\\ 
  1098 XG Amsterdam\\
  The Nederlands
    \\[0.3cm]
}

\begin{document}

\maketitle

\begin{abstract}
  In this short note we describe an alternative global version of the twisting 
  procedure used by Dolgushev to prove formality theorems.  
  This allows us to describe the maps of Fedosov resolutions, which are key factors 
  of the formality morphisms, in terms of a twist of the fiberwise quasi-isomorphisms
  induced by the local formality theorems proved by Kontsevich and Shoikhet. The key point
  consists in considering $L_\infty$-resolutions of the Fedosov resolutions obtained by Dolgushev 
  and an adapted notion of Maurer--Cartan element. This allows us to perform the twisting of the 
  quasi-isomorphism intertwining them in a global manner.
\end{abstract}


\newpage

%
%
\section{Introduction}

One of the most important results of deformation quantization is the
so-called formality theorem, due to Kontsevich \cite{Kontsevich2003}, proving the existence 
and classification of formal star products. More precisely,
the formality theorem provides  an
$L_\infty$-quasi-isomorphism from polyvectorfields to polydifferential
operators on Euclidean space. In \cite{Dolgushev2005, Dolgushev2005a}
Dolgushev proves the theorem for general manifolds $M$ by using Fedosov's formal geometric methods  \cite{Fedosov1996},
Kontsevich's quasi-isomorphism
\cite{Kontsevich2003} and the twisting procedure inspired by Quillen
\cite{Quillen69}. These techniques have also been used to prove
formality for Lie algebroids \cite{Calaque2005} and formality for chains
\cite{Dolgushev2006}.

The main goal of this paper is to give an alternative version 
of the twisting procedure used by Dolgushev. 
In \cite{Kontsevich2003} Kontsevich proved that there exists an $L_\infty$-quasi-isomorphism between dgla's
  \begin{equation}
    \mathscr{K}\colon 
    \Tpoly{}(\mathbb{R}^d)
    \longrightarrow 
    \Dpoly{}(\mathbb{R}^d).
  \end{equation} 
For a generic manifold $M$ the quasi-isomorphism $\mathscr{K}$ induces a fiberwise
quasi-isomorphism $\mathscr{U}$ between the Fedosov resolutions of $\Tpoly{}(M)$
and $\Dpoly{}(M)$. Twisting procedures are used by Dolgushev to obtain a formality
quasi-isomorphism by twisting the fiberwise
quasi-isomorphism induced by $\mathscr{U}$.
However, Dolgushev twists locally and checks consistency on overlapping charts, this means that the global quasi-isomorphism 
is not
described, a priori, as a twist of another morphism.  We are  interested in
presenting the quasi-isomorphism of Fedosov resolutions as a twist of the fiberwise map
globally.  This is possible if, first of all, one allows  for curvature in the
definition of an $L_\infty$-algebra and secondly one uses the notion of Maurer--Cartan elements adapted to a resolution of $L_\infty$-modules. This allows us to present the fiberwise
morphism $\mathscr{U}$ as a morphism of curved Lie algebras and
obtain the global map as a twist of $\mathscr{U}$ directly.
In complete analogy, starting with the quasi-isomorphism $\mathscr{S}$ of dgla-modules proved by 
Shoikhet in \cite{Shoikhet03} we are able to obtain the relevant map for chains 
as a twist of $\mathscr{S}$. The generalization to chains doesn't cost anything since we already phrase the case of 
cochains completely in terms of $L_\infty$-modules (and their resolutions as modules). 

In Dolgushev's approach the (global) $L_\infty$-algebra which one can twist is (globally) curved. 
Thus it makes no sense to speak of twisting a quasi-isomorphism. 
Our strategy consists in resolving the Fedosov resolutions obtained by Dolgushev
as $L_\infty$-modules. Then, $\mathscr{U}$ induces an $L_\infty$-morphism
of resolutions and we prove that the element by which one twists satisfies certain conditions. Locally we establish that the 
morphism of curved $L_\infty$-algebras is actually the twist of a quasi-isomorphism of flat $L_\infty$-algebras. Thus the use of a 
resolution allows us to consider whether $L_\infty$-morphisms of curved algebras are ``quasi-isomorphisms'', at least as far as 
twisting is concerned. 

%

\vspace{0.3cm}

The paper is organized as follows. In Section~\ref{sec:preliminaries}
we recall the language of $L_\infty$-algebras and in particular we 
present the notions of $L_\infty$-algebra and $L_\infty$-morphisms in the
presence of curvature. In this setting, we recall the twisting procedure and 
the effects that it has on $L_\infty$-algebras and $L_\infty$-morphisms.
Section~\ref{sec:global} contains the main result of this paper. First, 
we introduce the concept of Maurer--Cartan elements compatible with resolutions
of $L_\infty$-modules and we prove the theorem stated above. As a second step, we
apply this result to prove formality theorems for Hochschild cochains and chains.

\section*{Acknowledgments}
The authors are grateful to Stefan Waldmann 
for inspiring discussions.

%
%
\section{Preliminaries}
\label{sec:preliminaries}

%
%

In this section we recall the notions of $L_\infty$-algebras, $L_\infty$-modules, $L_\infty$-morphisms and their twists by Maurer--Cartan elements.  The idea of such
twisting procedures comes from Quillen's seminal work
\cite{Quillen69}. Proofs and details can be found in
\cite{Dolgushev2005,Dolgushev2005a,sexy1}.

Given a graded vector space $V^\bullet$ over $\mathbb{K}$ we shall
define the \emph{shifted} vector space, denoted by $V[k]^\bullet$, by
\begin{equation*}
  V[k]^\ell
  =
  V^{\ell+k}
\end{equation*}
We shall fix a field $\mathbb{K}$ of characteristic $0$.
Recall that a degree $+1$ coderivation $Q$ on the counital conilpotent
cocommutative coalgebra $S^c(\mathfrak{L})$ cofreely cogenerated by
the graded vector space $\mathfrak{L}[1]^\bullet$ over $\mathbb{K}$ is called
an \emph{$L_\infty$-structure} on the graded vector space
$\mathfrak{L}$ if $Q^2=0$.

We recall that $S^c(\mathfrak{L})$ can be realized by the symmetrized
deconcatenation product on the space
$\bigoplus_{n\geq0}\bigvee^n\mathfrak{L}[1]$, see e.g. \cite{sexy1}. Here $\bigvee^n\mathfrak{L}[1]$ is the space of
coinvariants for the usual action of $S_n$ (the symmetric group in $n$
letters) on $\otimes^n\mathfrak{L}[1]$ (keeping the grading in
mind). Any degree $+1$ coderivation $Q$ on $S^c(\mathfrak{L})$ is
uniquely determined by the components
\begin{equation}
  Q_n\colon \bigvee^n(\mathfrak{L}[1])\longrightarrow \mathfrak{L}[2]
\end{equation}
through the formula 
\begin{equation}
  Q(\gamma_1\vee\ldots\vee\gamma_n)
  =
  \sum_{k=0}^n\sum_{\sigma\in\mbox{\tiny Sh($k$,$n-k$)}}
  \epsilon(\sigma)Q_k(\gamma_{\sigma(1)}\vee\ldots\vee
  \gamma_{\sigma(k)})\vee\gamma_{\sigma(k+1)}\vee
  \ldots\vee\gamma_{\sigma(n)},
\end{equation} 
where Sh($k$,$n-k$) denotes the set of $(k, n-k)$ shuffles in $S_n$,
$\epsilon(\sigma)=\epsilon(\sigma,\gamma_1,\ldots,\gamma_n)$ is a sign
given by the rule
\begin{equation}
  \gamma_{\sigma(1)}\vee\ldots\vee\gamma_{\sigma(n)}=
  \epsilon(\sigma)\gamma_1\vee\ldots\vee\gamma_n
\end{equation} 
and we use the conventions that Sh($n$,$0$)=Sh($0$,$n$)$=\{\id\}$ and
that the empty product equals the unit.
Note that $Q_0(1)$ is of degree $1$ in $\mathfrak{L}[1]$ (thus of
degree $2$ in $\mathfrak{L}$).  The $Q^2=0$ condition can now be
expressed in terms of a quadratic equation in the $Q_n$.
\begin{example}[Curved Lie algebra]
\label{ex:curvedlie}
  Our main example of an $L_\infty$-algebra is that of a (curved) Lie
  algebra $(\mathfrak{L},R,\D,[\argument,\argument])$ by setting
  $Q_0(1)=-R$, $Q_1=-\D$,
  $Q_2(\gamma\vee\mu)=-(-1)^{|\gamma|}[\gamma,\mu]$ and $Q_i=0$ for
  all $i\geq 3$. 
\end{example}
In the following we call the $L_\infty$-algebras with $Q_0=0$
\emph{flat} $L_\infty$-algebras.  For our purposes, we need to
consider $L_\infty$-algebras $\mathfrak{L}$ that are equipped with a
decreasing filtration
\begin{equation*}
  \mathfrak{L}=\mathcal{F}^0\mathfrak{L}\supset\mathcal{F}^1\mathfrak{L}\supset\ldots\supset\mathcal{F}^k\mathfrak{L}\supset\ldots,
\end{equation*}
that respects the $L_\infty$-structure and which is moreover
\emph{complete}, i.e.
\begin{equation*}
  \bigcap_k\mathcal{F}^k\mathfrak{L}=\{0\}.
\end{equation*}
This yields an associated complete metric topology and we consider
convergence of infinite sums in terms of this topology. 

\vspace{0.3cm} 

\noindent Suppose
$(\mathcal{L},Q)$ and $(\widetilde{\mathcal{L}},\widetilde{Q})$ are
two $L_\infty$-algebras.  A degree $0$, filtration respecting,
counital coalgebra morphism
\begin{equation*}
  F\colon 
  S^c(\mathfrak{L})
  \longrightarrow 
  S^c(\widetilde{\mathfrak{L}})
\end{equation*}
such that $FQ = \widetilde{Q}F$ is said to be an
\emph{$L_\infty$-morphism}.
%
A coalgebra morphism $F$ from $S^c(\mathfrak{L})$ to
$S^c(\widetilde{\mathfrak{L}})$ is uniquely determined by its
components (also called \emph{Taylor coefficients})
\begin{equation*}
  F_n\colon \bigvee^n(\mathfrak{L}[1])\longrightarrow \widetilde{\mathfrak{L}}[1],
\end{equation*}
where $n\geq 1$. Namely, we set $F(1)=1$ and use the formula
\begin{equation*}
  F(\gamma_1\vee\ldots\vee\gamma_n)=
\end{equation*}
\begin{equation}
\label{coalgebramorphism}
  \sum_{p\geq1}\sum_{\substack{k_1,\ldots, k_p\geq1\\k_1+\ldots+k_p=n}}
  \sum_{\sigma\in \mbox{\tiny Sh($k_1$,..., $k_p$)}}\frac{\epsilon(\sigma)}{p!}
  F_{k_1}(\gamma_{\sigma(1)}\vee\ldots\vee\gamma_{\sigma(k_1)})\vee\ldots\vee 
  F_{k_p}(\gamma_{\sigma(n-k_p+1)}\vee\ldots\vee\gamma_{\sigma(n)}),
\end{equation}
where Sh($k_1$,...,$k_p$) denotes the set of $(k_1,\ldots,
k_p)$-shuffles in $S_n$ (again we set Sh($n$)$=\{\id\}$).
Given an $L_\infty$-morphism of \emph{flat} $L_\infty$-algebras
$\mathfrak{L}$ and $\widetilde{\mathfrak{L}}$, we obtain the map of
complexes
\begin{equation*}
  F_1\colon (\mathfrak{L},Q_1)\longrightarrow (\widetilde{\mathfrak{L}},\widetilde{Q}_1).
\end{equation*}
The $L_\infty$-morphism $F$ is called an
\emph{$L_\infty$-quasi-isomorphism} if this map $F_1$ is a
quasi-isomorphism of complexes.

Let $\mathfrak{L}$ be an $L_\infty$-algebra. Then an 
\emph{$L_\infty$-module} over $\mathfrak{L}$ is a 
graded vector space $\mathfrak{M}$ equipped with a
square-zero degree $+1$ coderivation 
$\varphi$ on the cofree 
$S^c(\mathfrak{L})$-comodule $S^c(\mathfrak{L})
\tensor \mathfrak{M}$ cogenerated by $\mathfrak{M}$.
Note that as for a coderivation on 
$S^c(\mathfrak{L})$ a coderivation 
$\varphi$ on $S^c(\mathfrak{L})\otimes\mathfrak{M}$ is given by the components 
\begin{equation}
\varphi_n\colon \bigvee^n\mathfrak{L}[1]\otimes \mathfrak{M}\rightarrow \mathfrak{M}[1]
\end{equation}
through the formula 
\begin{equation}
\begin{gathered}
\varphi(\gamma_1\vee\ldots\vee\gamma_n\otimes m)=
\\
Q(\gamma_1\vee\ldots\vee\gamma_n)\otimes m +
\sum_{k=0}^n\sum_{\sigma\in\mbox{Sh}(k,n-k)}\epsilon'(\sigma)\epsilon(\sigma)
\gamma_{\sigma(1)}\vee\ldots\vee\gamma_{\sigma(k)}
\otimes \varphi_{n-k}(\gamma_{\sigma(k+1)}
\vee\ldots\vee\gamma_{\sigma(n)}\otimes m),
\end{gathered}
\end{equation}
where $\epsilon'(\sigma)=\epsilon'
(\sigma,\gamma_1\vee\ldots\vee\gamma_n)=
(-1)^{\sum_{i=1}^k|\gamma_{\sigma(i)}|}$.
So for instance we find that $\varphi
(\gamma\otimes m)$ is given by $Q(\gamma)\otimes m +
1\otimes \varphi_1(\gamma\otimes m) +
(-1)^{|\gamma|}\gamma\otimes \varphi_0(1\otimes m)$. 
The square-zero condition yields conditions quadratic
in the $\varphi_n$ and $Q_n$, for example
\begin{equation*}
\begin{gathered}
\varphi_0(1\otimes \varphi_0(1\otimes m))+
\varphi_1(Q_0(1)\otimes m)=0\\
\varphi_0(1\otimes \varphi_1(\gamma\otimes m))
+(-1)^{|\gamma|}\varphi_1(\gamma\otimes 
\varphi_0(1\otimes m))+\varphi_1
(Q_1(\gamma)\otimes m)+\varphi_2(Q_0(1)\vee\gamma\otimes m)=0.
\end{gathered}
\end{equation*}
Note that if $Q_0(1)=0$ then by identifying 
$\bigvee^0\mathfrak{L}[1]\otimes\mathfrak{M}$ with $\mathfrak{M}$ we obtain that 
$\varphi_0$ is a differential on $\mathfrak{M}$. 
\begin{example}
  The second most basic example of 
   an $L_\infty$-module
  consists of a dg module $(\mathfrak{M}, \mathrm{b},\rho)$ 
  over a dgla $(\mathfrak{L}, \D, [\argument, \argument])$.
  In this case we have 
  \begin{equation}
  \begin{gathered}
      \varphi_0 (v) = -\mathrm{b}v 
      \\
      \varphi_1(\gamma\vee v) = -(-1)^{|\gamma|}\rho (\gamma)v\quad\text{and}\\
      \varphi_k=0\hspace{0.2cm}\mbox{for all}\hspace{0.2cm}k\geq 2,
      \end{gathered}
  \end{equation}
  for $v \in \mathfrak{M}$ and $\gamma \in \mathfrak{L}$, where $\rho$ is the action
  of $\mathfrak{L}$ on $\mathfrak{M}$.
\end{example}

\begin{example}[Morphism of $L_\infty$-algebras]
  \label{morphmodule}
  Suppose $F\colon \mathfrak{L}\rightarrow \mathfrak{K}$ is an 
  $L_\infty$-morphism. This induces the structure of $L_\infty$-module over $\mathfrak{L}$ on $\mathfrak{K}$. 
  Namely, we consider the module structure with components 
  $\varphi_k$ given by 
  \begin{equation}
  	\varphi_k(\gamma_1\vee\ldots\vee\gamma_k\otimes m)
    =
    \pr_{\mathfrak{K}}\left(Q_{\mathfrak{K}}(F(\gamma_1\vee\ldots\vee\gamma_k)\vee m\right)
  \end{equation}
\end{example}
Let $\mathfrak{L}$ be an $L_\infty$-algebra and $(\mathfrak{M}, \varphi)$,
$(\widetilde{\mathfrak{M}}, \tilde{\varphi})$ be $L_\infty$-modules over $\mathfrak{L}$.
Then a morphism $F$ from the comodule $S^c(\mathfrak{L})\tensor \mathfrak{M}$ to the
comodule $S^c(\mathfrak{L})\tensor \widetilde{\mathfrak{M}}$ is said to be an \emph{$L_\infty$-morphism}
if it satisfies the condition:
\begin{equation*}
	F\varphi
    =
    \tilde{\varphi}F.
\end{equation*}
As before a degree $0$ comodule morphism $F\colon 
S^c(\mathfrak{L})\otimes \mathfrak{M}\rightarrow S^c(\mathfrak{L})\otimes \widetilde{\mathfrak{M}}$ 
is given by components 
\begin{equation*}
F_n\colon \bigvee^n\mathfrak{L}[1]\otimes\mathfrak{M}\longrightarrow \widetilde{\mathfrak{M}}
\end{equation*}
through the formula 
\begin{equation} 
\begin{gathered}
F(\gamma_1\vee\ldots\vee\gamma_n\otimes m)=\\
\sum_{k=0}^n\sum_{\sigma\in\mbox{Sh}(k,n-k)}\epsilon(\sigma)\gamma_{\sigma(1)}\vee\ldots\vee\gamma_{\sigma(k)}\otimes F_{n-k}(\gamma_{\sigma(k+1)}\vee\ldots\vee\gamma_{\sigma(n)}\otimes m)
\end{gathered}
\end{equation}
In particular, in the case that $\mathfrak{L}$ is flat, $F$ is a \emph{quis of $L_\infty$-modules} if
the zero-th
component $F_0$  is a quis of complexes. 

\begin{example}[Morphism of $L_\infty$-algebras]
\label{morphmodulemorph}
  Suppose 
  \begin{equation}
      \begin{tikzcd}
        \mathfrak{L}\arrow[r, "G"]\arrow[swap,"H",d]& \mathfrak{G}\\
        \mathfrak{H}\arrow[swap,ur,"F"]&\\
      \end{tikzcd}
  \end{equation}
  is a commuting diagram of morphisms of $L_\infty$-algebras.  
  Then we may equip $\mathfrak{H}$ and $\mathfrak{G}$ with the
   $\mathfrak{L}$-module structure as in Example~\ref{morphmodule}
   and we find that the map $\mathcal{F}\colon \mathfrak{G}
   \rightarrow \mathfrak{H}$ given by 
    \begin{equation}\mathcal{F}_n(\gamma_1\vee\ldots\vee\gamma_n\otimes m)=
    \pr_{\mathfrak{G}}(F(H(\gamma_1\vee\ldots\vee\gamma_n)\vee m))\end{equation}
    is a morphism of $L_\infty$-modules. Note in particular that $\mathcal{F}_0(1\otimes m)=F_1(m)$ so, if $\mathfrak{L}$, $\mathfrak{G}$ and $\mathfrak{H}$ are flat, then $\mathcal{F}$ is a quasi-isomorphism if and only if $F$ is a quasi-isomorphism. 
\end{example}
Let $\pi\in\mathcal{F}^1\mathfrak{L}[1]^0$, by 
direct computations we find that the element
\begin{equation*}
    \exp(\pi)
    :=
    \sum_{n=0}^\infty\frac{\pi^k}{k!}\in S^c(\mathfrak{L})
\end{equation*}
is well-defined, invertible and group-like. As a consequence one can prove 
the following claim.
\begin{lemma}
	\label{exppi}
    The pair $\mathfrak{L}^\pi = (\mathfrak{L}, Q^\pi)$ with $
    Q^\pi(a):=\exp(-\pi)\vee Q(\exp(\pi)\vee a) $ is still an
    $L_\infty$-algebra.
\end{lemma}
\begin{example}[Twisted curved Lie algebra]
  In the case of a curved Lie algebra $(\mathfrak{L}, R,\D,
  [\argument,\argument])$ we find the twisted curved Lie algebra
  $(\mathfrak{L}, R^\pi, \D + [\pi,\argument],[\argument,\argument])$,
  where
  \begin{equation*}
    R^\pi
    :=
    R+\D\pi + \frac{1}{2}[\pi,\pi].
  \end{equation*}
\end{example}

Given an $L_\infty$-algebra $(\mathfrak{L}, Q)$, an element $\pi\in
\mathcal{F}^1\mathfrak{L}[1]^0$ is called a \emph{Maurer-Cartan (MC) element} if it satisfies the following equation
\begin{equation}
 \label{MCdef}
  \sum_{n=0}^\infty \frac{Q_n(\pi^n)}{n!}=0.
\end{equation}
Note that this is equivalent to $Q^\pi_0=0$, and it is equivalent to $Q(\exp(\pi))=0$. 
For a dgla $(\mathfrak{L}, \D,[\argument,\argument])$ the definition
above boils down to the usual Maurer--Cartan equation.
If we have similarly a curved Lie algebra with curvature $-R$ it comes
down to the non-homogeneous equation
\begin{equation*}
  \D\pi+\frac{1}{2}[\pi,\pi]
  =
  R.
\end{equation*}
\begin{lemma}
 \label{lem:twistprop}
 Given an L$_\infty$-morphism $F$ from $\mathfrak{L}$ 
 to $\widetilde{\mathfrak{L}}$ and an element
  $\pi\in\mathcal{F}^1\mathfrak{L}[1]^0$, we define the $F$-associated element $\pi_F\in\mathcal{F}^1
  \widetilde{\mathfrak{L}}[1]^0$ by the formula
  \begin{equation*}
    \pi_F:=\sum_{n=1}^\infty\frac{F_n(\pi^n)}{n!}.
  \end{equation*}
  Then:
  \begin{lemmalist}
  \item\label{lem:twistprop!1}  we have
    \begin{equation*}
      F(\exp(\pi))=\exp(\pi_F);
    \end{equation*}
  \item\label{lem:twistprop!2} if $\pi$ is an MC element, then $\pi_F$ is also an MC element;
  \item\label{lem:twistprop!3}   the $\pi$-twist $F^\pi\colon\mathfrak{L}
  ^\pi\rightarrow \widetilde{\mathfrak{L}}^{\pi_F}$
  of $F$ defined by
    \begin{equation*}
       F^\pi(a)
      :=
      \exp(-\pi_F)\vee F(\exp(\pi)\vee a)
    \end{equation*}
    is an $L_\infty$-morphism;
  \item\label{lem:twistprop!4} if $F$ is an $L_\infty$-morphism such that
    the induced morphisms
    \begin{equation*}
      F|_{\mathcal{F}^k\mathfrak{L}}
      \colon 
      \mathcal{F}^k\mathfrak{L}\longrightarrow \mathcal{F}^k\widetilde{\mathfrak{L}}
    \end{equation*}
    are $L_\infty$-quasi-isomorphisms for all $k$ and
    $\pi\in\mathcal{F}^1\mathfrak{L}[1]^0$ is an MC 
    element, then the $\pi$-twist $F^\pi$ of $F$ is 
    also a quasi-isomorphism.
  \end{lemmalist}
\end{lemma}
Here \refitem{lem:twistprop!1}-\refitem{lem:twistprop!3} are proved through simple computations 
and for \refitem{lem:twistprop!4} we refer to \cite[Prop. 1]{Dolgushev2005}.
\begin{remark}
\label{rem:comptwist}
  Note that, given $L_\infty$-morphisms $F$ and $G$ from
  $\mathfrak{L}$ to $\mathfrak{L}'$ and $\mathfrak{L}'$ to
  $\mathfrak{L}''$ respectively and the elements $\pi, B\in
  \mathcal{F}^1\mathfrak{L}[1]^0$, we have that
  \begin{equation*} 
    (Q^{\pi})^B=Q^{\pi+B}=(Q^B)^\pi,\hspace{0.3cm} (F^\pi)^B=F^{\pi+B}=(F^B)^\pi,
  \end{equation*}
  \begin{equation*}
    \pi_F+B_{F^\pi}=(\pi+B)_F=B_F+\pi_{F^B}\hspace{0.3cm}\mbox{and}\hspace{0.3cm}(\pi_F)_G=\pi_{G\circ F}.
  \end{equation*}
\end{remark}
The results recalled in Lemma~\ref{lem:twistprop} can be also obtained 
in the setting of $L_\infty$-modules. More precisely, 
let $(\mathfrak{M},\varphi)$ be an $L_\infty$-module 
over the $L_\infty$-algebra  $(\mathfrak{L},Q)$. 
The same graded (filtered) vector space $S^c(\mathfrak{L})
\otimes \mathfrak{M}$ that forms the cofree comodule cogenerated by $\mathfrak{M}$ 
is simultaneously the free module generated by 
$\mathfrak{M}$. Thus, since $S^c(\mathfrak{L})$ is 
commutative, every element $a\in S^c(\mathfrak{L})$
defines a module morphisms of $S^c(\mathfrak{L})
\otimes \mathfrak{M}$ that we will denote by
concatenation. Then it is easy to see that given 
$\pi\in\mathcal{F}^1\mathfrak{L}[1]^0$ we obtain the 
twisted $L_\infty$-module $(\mathfrak{M}^
\pi,\varphi^\pi)$ over the twisted $L_\infty$-algebra
$\mathfrak{L}^\pi$ by setting $\mathfrak{M}^
\pi:=\mathfrak{M}$ and 
\begin{equation}
\label{eq:varphipi}
	\varphi^\pi(X)
    =
    e^{-\pi}\varphi(e^\pi X).
\end{equation}
Similarly if $F\colon (\mathfrak{M}, \varphi)
\rightarrow(\widetilde{\mathfrak{M}},
\tilde{\varphi})$ is an $L_\infty$-morphism we obtain
the $L_\infty$-morphism $F^\pi\colon \mathfrak{M}^
\pi\rightarrow \widetilde{\mathfrak{M}}^\pi$ given by
\begin{equation}
	F^\pi(X)
    =
    e^{-\pi}F(e^\pi X).
\end{equation}
As before we find that if $F$ is an $L_\infty$-quis (respecting filtrations) then 
$F^\pi$ is also an $L_\infty$-quis.
\begin{example}[Morphism of $L_\infty$-algebras]\label{morphmoduletwist} 
Consider Example \ref{morphmodule}, in this case the 
above twisting leads a priori to two different modules.
Namely we can either twist the module $\varphi$ 
obtained from the $L_\infty$algebra morphism $F$ or we 
can first twist $F$ to obtain a module structure. A 
straightforward check shows that these two modules 
coincide. 

\vspace{0.3cm} 

Similarly we may consider Example~\ref{morphmodulemorph}. In this case we can either twist the morphism $\mathcal{F}$ or 
note that $F^{\pi_H}\circ H^\pi=(F\circ H)^\pi=G^\pi$ to obtain the morphism induced by $F^{\pi_H}$ directly. By the 
previous paragraph these are two morphisms between the same modules. In fact they coincide.  
\end{example}


\section{A global approach to twisting procedure}
\label{sec:global}

In \cite{Dolgushev2005,Dolgushev2006} Dolgushev uses the twisting method to obtain a certain quasi-isomorphism. This 
method proceeds roughly as follows. One starts with two flat $L_\infty$-algebras $\mathfrak{K}$ and $\mathfrak{L}$ and one 
wants to find an $L_\infty$-quasi-isomorphism $F$ between them. It is often too hard to construct such a morphism 
directly, but one can make use of the twisting procedure of $L_\infty$-algebras.

Let us recall the result of this procedure in order to make this point explicit. The result of the twisting procedure is 
that the elements $\pi\in \mathcal{F}^1\mathfrak{L}[1]^0$ parametrize a family of $L_\infty$-algebras $\mathfrak{L}^\pi$ 
and the subset of Maurer--Cartan elements $\mathrm{MC}(\mathfrak{L})$ parametrize a subfamily characterized by flatness.
The same is of course true for $\mathfrak{K}$. Given a morphism $G\colon\mathfrak{L}^\pi\longrightarrow 
\mathfrak{K}^{\pi'}$ from an algebra in one family to an algebra in the other we obtain morphisms from each algebra in 
family parametrized by $\mathcal{F}^1\mathcal{L}[1]^0$ to an algebra in the family parametrized by $\mathcal{F}^1
\mathfrak{K}[1]^0$. Thus each such morphism also induces a family of morphisms parametrized by $\mathcal{F}^1\mathfrak{L}
[1]^0$. Moreover, if we start with $\pi\in \mathrm{MC}(\mathfrak{L})$ and $G$ is a quasi-isomorphism (respecting 
filtrations) then the subfamily parametrized by $\mathrm{MC}(\mathfrak{L})$ consists of quasi-isomorphisms. Note that this 
is also of significance to the rest of the morphisms in the family; in a sense these are ``quasi-isomorphisms'' of curved
$L_\infty$-algebras. 

The idea used by Dolgushev, when $\mathfrak{L}$ and $\mathfrak{K}$ are given by the Fedosov resolutions of polyvector 
fields and polydifferential operators respectively (see section \ref{sec:applications}), is to look for a quasi-
isomorphism in the family parametrized by $\mathcal{F}^1\mathfrak{L}[1]$ that gets twisted into a map from $\mathfrak{L}$ 
to $\mathfrak{K}$, thus showing that these algebras are quasi-isomorphic. Such a quasi-isomorphism is made readily 
available in his case by considering  Kontsevich's map from Theorem~\ref{formalformality} applied fiberwise. A problem 
that arises is that the $L_\infty$-algebras between which Kontsevich's map operates are not flat. 

Dolgushev resolves this issue by first working locally and showing that these local solutions glue appropriately. This is 
not preferable since the resulting quasi-isomorphism is not explicitly realized as a twist. In this section we construct 
the tools needed to perform the twisting in an explicitly global manner and 
apply the method to establish formality of polydifferential operators on an arbitrary manifold given the formal 
formality of Theorem~\ref{formalformality}. 
The basic idea is to replace the ``glueing'' argument of Dolgushev by a resolution of $L_\infty$-modules 
given by a cover. We thus show that the fiberwise application of Kontsevich's map actually yields one of the ``quasi-
isomorphisms" of curved $L_\infty$-algebras mentioned above. This method has the added benefit of working mutatis mutandis 
for the case of 
chains, which we will also exemplify. 

\vspace{0.3cm}
\noindent
Suppose $\mathfrak{L}$ is an $L_\infty$-algebra and $\mathfrak{M}$ is an 
$L_\infty$-module over $\mathfrak{L}$. Consider a resolution 
\begin{equation}
  0
  \rightarrow
  \left(\mathfrak{M}, \varphi\right)
  \stackrel{F}{\longrightarrow} 
  \left(\mathfrak{M}^0, \varphi^0\right)
  \stackrel{\partial^0}{\longrightarrow}
  \left(\mathfrak{M}^1,\varphi^1\right)
  \stackrel{\partial^1}{\longrightarrow}\ldots 
\end{equation}
of $\mathfrak{M}$, which we denote by $F\colon\mathfrak{M}\rightarrow 
\left(\mathfrak{M}^\bullet,\partial^\bullet\right)$ or simply $F$. 
Note that this means that the graded vector spaces $\mathfrak{M}^i$ are all $L_\infty$-modules 
over $\mathfrak{L}$ and the maps $\partial^i$ are $L_\infty$-morphisms. 
\begin{definition}[Resolution adapted MC elements]
  The resolution adapted MC elements are those MC elements $\pi$ of $\mathfrak{L}$ that have the 
  property that the induced complex 
  \begin{equation}
    0
    \rightarrow 
    \mathrm{H} \left(\mathfrak{M},\varphi^\pi_0\right)
    \stackrel{F^\pi_0}{\longrightarrow}
    \mathrm{H}\left(\mathfrak{M}^0,(\varphi^0)^\pi_0\right)\stackrel{\partial^\pi_0}{\longrightarrow}
    \mathrm{H}\left(\mathfrak{M}^1,(\varphi^1)^\pi_0\right)\stackrel{\partial^\pi_1}{\longrightarrow}\ldots
  \end{equation}
  is acyclic. The set of resolution adapted MC elements is denoted by $\mathrm{MC}\left(F\right)$.
\end{definition}
\begin{definition}[Morphisms of Resolutions]
  Given resolutions $F\colon \mathfrak{M}\rightarrow\mathfrak{M}^\bullet$ and 
  $G\colon\mathfrak{N}\rightarrow\mathfrak{N}^\bullet$ of $L_\infty$-modules over 
  $\mathfrak{L}$, a series of $L_\infty$-morphisms $\mathscr{U}\colon \mathfrak{M}\rightarrow\mathfrak{N}$ 
  and $\mathscr{U}^{\bullet}\colon\mathfrak{M}^\bullet\rightarrow\mathfrak{N}^\bullet$ is an 
  $L_\infty$-morphism from $F$ to $G$ if the following diagram 
  \begin{equation}
      \begin{tikzcd}
        0\arrow[r]& \mathfrak{M}
        \arrow[r, "F"]\arrow[d, swap, "\mathscr{U}"] 
        & \mathfrak{M}^0
        \arrow[r, "\partial_F^0"]\arrow[d, swap, "\mathscr{U}^0"] 
        & \mathfrak{M}^1
        \arrow[r, "\partial_F^1"]\arrow[d,  "\mathscr{U}^1"]
        & \ldots 
        \\
        0\arrow[r]& \mathfrak{N} \arrow[r, "G"]
        & \mathfrak{N}^0 \arrow[r,"\partial^0_G"]
        & \mathfrak{N}^1 \arrow[r,"\partial_G^1"]
        & \ldots
      \end{tikzcd}
  \end{equation}
  commutes.
\end{definition}
%
%
\begin{proposition}
\label{prop:key}
  Suppose $\mathscr{U}\colon F\rightarrow G$ is an $L_\infty$-morphism of resolutions from a 
  resolution of the $\mathfrak{L}$-module $\mathfrak{M}$ to a resolution of the 
  $\mathfrak{L}$-module $\mathfrak{N}$. Suppose further that $\pi\in \mathrm{MC}(F)\cap \mathrm{MC}(G)$ 
  and $(\mathscr{U}^n)^\pi$ is a quasi-isomorphism for all $n\geq 0$. 
  Then $\mathscr{U}^\pi$ is an $L_\infty$-quasi-isomorphism. 
\end{proposition}
\begin{proof} 
	Since $\pi$ is adapted to both $F$ and $G$ we find that taking the cohomology with 
    respect to the differentials on $(\mathfrak{M}^\bullet)^\pi$ and $(\mathfrak{N}^\bullet)^\pi$ 
    yields the commutative diagram 
    \begin{equation}
      \begin{tikzcd}
      0\arrow[r]& \mathrm{H}(\mathfrak{M})
        \arrow[r, "\mathrm{H}F^\pi_0"]\arrow[d, swap, "\mathrm{H}\mathscr{U}^\pi_0"] 
        & \mathrm{H}(\mathfrak{M}^0)
        \arrow[r, "\mathrm{H}(\partial_F^0)^\pi_0"]\arrow[d, swap, "\mathrm{H}(\mathscr{U}^0)^\pi_0"] 
        & \mathrm{H}(\mathfrak{M}^1)
        \arrow[r, "\mathrm{H}(\partial_F^1)^\pi_0"]\arrow[d,  "\mathrm{H}(\mathscr{U}^1)^\pi_0"]
        & \ldots 
        \\
        0\arrow[r]& \mathrm{H}(\mathfrak{N}) \arrow[r, "\mathrm{H}G^\pi_0"]
        & \mathrm{H}(\mathfrak{N}^0) \arrow[r,"\mathrm{H}(\partial^0_G)^\pi_0"]
        & \mathrm{H}(\mathfrak{N}^1) \arrow[r,"\mathrm{H}(\partial_G^1)^\pi_0"]
        & \ldots
      \end{tikzcd}
  \end{equation}
  with exact rows and such that the the vertical arrows to 
  right of $\mathrm{H}\mathscr{U}^\pi_0$ are all isomorphisms. 
  Thus $\mathrm{H}\mathscr{U}_0^\pi$ is also an isomorphism and 
  $\mathscr{U}^\pi$ is a quasi-isomorphism. 
\end{proof}
%


\subsection{Applications: formality theorems}
\label{sec:applications}

In the following section we apply the above result to obtain a proof of 
formality of Hochschild cochains and chains. The main point consists in showing 
that the Fedosov resolutions obtained by Dolgushev in \cite{Dolgushev2005a,Dolgushev2006}
form resolutions of $L_\infty$-modules. Thus, formality maps can be obtained as a twisted
morphism of resolutions via Prop.~\ref{prop:key}.

\subsubsection{Formality for Hochshild cochains}

As a first step we here need to recall the resolutions obtained in \cite{Dolgushev2005,Dolgushev2005a}
of the dgla's of poly-vector fields and poly-differential operators on a
generic manifold. A more detailed discussion can also be found in \cite{sexy1}.

Let us denote the formal neighborhood at $0\in \R^d$ by $\Rf$. The
smooth functions $\Cinfty(\Rf)$ on $\Rf$ are given by the algebra
\begin{equation*}
  \Cinfty(\Rf)
  :=
  \varprojlim_{k\rightarrow \infty} \Cinfty(\R^d)/\mathcal{I}_0^k,
\end{equation*} 
where $\mathcal{I}_0$ denotes the ideal of functions vanishing at
$0\in\R^d$.  Note that $\Cinfty(\Rf)$ comes equipped with the complete
decreasing filtration
\begin{equation*}
   \Cinfty(\Rf)\supset\mathcal{I}_0\supset\mathcal{I}_0^2\supset\ldots
 \end{equation*}
and corresponding (metric) topology.  The Lie algebra of continuous
derivations of $\Cinfty(\Rf)$ is denoted by $\Tpoly{0}(\Rf)$. Setting
$\Tpoly{-1} := \Cinfty(\Rf)$ we obtain the Lie--Rinehart pair
$(\Tpoly{-1},\Tpoly{0})$ and the graded vector space
\begin{equation*}
  \Tpoly{} (\Rf)
  :=
  \bigoplus_{k\geq -1}\Tpoly{k} (\Rf),
\end{equation*}
where $\Tpoly{k} (\Rf) := \Anti^{k+1}\Tpoly{0} (\Rf)$ for $k\geq 0$.
Here the tensor product is understood to be over $\Tpoly{-1}(\Rf)$ and
completed.
The Lie bracket $\Schouten{\argument,\argument}$ on $\Tpoly{0}(\Rf)$
extends to a graded Lie algebra structure on $\Tpoly{}(\Rf)$.

The universal enveloping algebra of the Lie-Rinehart pair
$(\Tpoly{-1}(\Rf),\Tpoly{0}(\Rf))$ is denoted by
$\Dpoly{0}(\Rf)$. 
%
We extend the algebra structure in the obvious (componentwise) way to
\begin{equation*}
  \Dpoly{}(\Rf)
  :=
  \bigoplus_{k\geq -1}\Dpoly{k}(\Rf),
\end{equation*} 
where $\Dpoly{-1} (\Rf):= \Tpoly{-1}(\Rf)$ and $\Dpoly{k} (\Rf) :=
\left(\Dpoly{0}(\Rf)\right)^{\otimes k+1}$. Here, as before, the
tensor product is understood to be over $\Dpoly{-1}(\Rf)$ and
completed. This allows us to define the Gerstenhaber bracket
$[\argument,\argument]_G$ which endows $\Dpoly{}
(\Rf)$ with a graded Lie algebra structure. 
%
%
\begin{theorem}[Kontsevich\cite{Kontsevich2003}]
  \label{formalformality}
  There exists an $L_\infty$-quasi-isomorphism between dgla's
  \begin{equation}
    \mathscr{K}\colon 
    \left(\Tpoly{}(\Rf),0,\Schouten{\argument,\argument}\right)
    \longrightarrow 
    \left(\Dpoly{}(\Rf),\partial,[\argument,\argument]_G\right)
  \end{equation} 
  where $\partial = [\mu,\argument]_G$ for $\mu=1\otimes 1\in \Dpoly{1}(\Rf)$. Moreover 
  \begin{enumerate} 
  \item  \label{formalformality1} $\mathscr{K}$ is $\group{GL}(d,\R)$ invariant;
  \item  \label{formalformality2} $\mathscr{K}_n(X_1\vee \ldots\vee X_n)=0$ for all $X_i\in \Tpoly{0}(\Rf)$ and $n>1$;
  \item  \label{formalformality3} $\mathscr{K}_n(X\vee Y_2\vee\ldots\vee Y_n)=0$ for all 
  $Y_i\in\Tpoly{}(\Rf)$ and $n\geq 2$ whenever $X\in\Tpoly{0}(\Rf)$ is induced by the action of $\mathfrak{gl}(d,\R)$. 
  \end{enumerate}
\end{theorem}
The definitions of the dgla's $\Tpoly{}(\Rf)$ and $\Dpoly{}(\Rf)$ go
through mutatis mutandis to define the dgla's $ \Tpoly{}(M)$ and $
\Dpoly{}(M)$ on a generic manifold $M$, starting from the Lie-Rinehart pair $\left(\Cinfty(M),
\Secinfty(TM)\right)$. Note that the resulting spaces $
\Dpoly{k}(M)$ can be identified with the vector space of
polydifferential operators of order $k+1$.
%

The bundle $ \FibT{}$ of \emph{formal fiberwise polyvector}
  fields is the bundle over $M$ with fiber $\Tpoly{}(\Rf)$
  associated with the principal bundle of general linear frames in
  $TM$. Similarly, for the bundle $ \FibD{}$ of \emph{formal fiberwise
  polydifferential operators}.
The differential forms with values in these bundles form the dgla's $(
\FormsT{}{M}{} , 0 , \Schouten{\argument , \argument})$ and $(
\FormsD{}{M}{} , \partial , [\argument , \argument]_G)$ respectively. The dgla structure may be 
induced from the dgla structures on the fibers since
it is compatible with the general linear action.

Note that $\Tpoly{}(\Rf)\hookrightarrow\Dpoly{}(\Rf)$ 
by the usual anti-symmetrization map and thus also 
$\FormsT{}{M}{}\hookrightarrow\FormsD{}{M}{}$. 
So, any element $A\in\FormsT{\ell}{M}{k}$ defines an 
operator $\Schouten{A,\argument}$ of degree $k+\ell$ on $\FormsT{}{M}{}$ and 
$\FormsD{}{M}{}$. 
Let us consider the $\FibT{}$ analog of the fundamental one-form used in the Fedosov
construction \cite[Def. 1.3.1]{Fedosov1996}, which we denote by $A_{-1}$ and set 
$\delta=\Schouten{A_{-1},\argument}$. In local coordinates $(x^1,\ldots, x^n)$ we have 
$A_{-1} = \sum_{i=1}^n\partial_{\hat{x}^i} \D x^i$ and 
thus $[A_{-1},A_{-1}]=0$ which implies that 
$\delta^2=0$ and yields the dgla's 
$(\FormsT{}{M}{},-\delta, \Schouten{\argument,\argument})$ 
and $(\FormsD{}{M}{},\partial-\delta,[\argument,\argument]_G)$. 
%

Following the idea of Fedosov, one changes the 
differential $\delta$ by adding terms of higher 
degree in the fiberwise grading. This way the 
cohomology remains the same, but with the correct
dgla structure. Since $\delta$ is of fiberwise degree
$-1$ we start by adding a linear connection $\conn$ 
as a degree $0$ term. It can be checked that $\conn A_{-1}$ 
coincides with the $\FibT{}$ equivalent of the 
torsion 2-form of $\conn$. Thus by picking a torsion-free
connection $\conn$ we find that  
$(-\delta+\conn)^2=\conn^2$.  Since
there is no reason to assume that we can find $\conn$ such that
$\conn^2=0$ we correct $-\delta+\conn$ by an inner derivation and make
the ansatz
\begin{equation}
\label{eq:fedosovD}
  D
  :=
  -\delta + \conn + \Schouten{A, \argument}
\end{equation}
with $A \in \FormsT{1}{M}{0}$ (of fiberwise degree greater than $1$). It can be proved that one may 
always find $A$ such that $D^2=0$. In fact such $A$
is unique if one adds the normalization 
$\delta^{-1}A=0$. Such $A$ also yields the maps from
$\Tpoly{}(M)$ and $\Dpoly{}(M)$ to $\FormsT{}{M}{}$ 
and $\FormsD{}{M}{}$ respectively that yield the 
following theorem.
\begin{theorem}[Fedosov Resolutions, Dolgushev \cite{Dolgushev2005,Dolgushev2005a}]
  \label{FedRes}
 There exist dgla quasi-isomorphisms 
  \begin{equation*}
    \lambda_D
    \colon 
    \left(\Dpoly{}(M),\partial\right)\longrightarrow			\left(\FormsD{}{M}{},\partial + D\right)
    \quad\text{and}\quad
    \lambda_T
    \colon 
    \left(\Tpoly{}(M),0\right)\longrightarrow \left(\FormsT{}{M}{},D\right).
  \end{equation*} 
\end{theorem}
\begin{example}[Formality for $\R^d$]
  \label{formalityRd}
  In this example we generalize the result of Theorem
  \ref{formalformality} from $\Rf$ to $\R^d$ by using the twisting
  procedure discussed in Section~\ref{sec:preliminaries}.  Note that we are looking for an
  $L_\infty$-quasi-isomorphism
  \begin{equation*}
    \mathscr{U}^{A_{-1}}\colon(\Omega(\R^d;\FibT{}),D)\longrightarrow (\Omega(\R^d,\FibD{})
    ,\partial+D),
  \end{equation*}
  since this would complete the diagram
  \begin{equation}\label{seqquisRd}
    (\Tpoly{}(\R^d),0)\stackrel{\lambda_T}{\longrightarrow}
    (\Omega(\R^d;\FibT{}),D)\stackrel{\mathscr{U}^{A_{-1}}}{\longrightarrow}
    (\Omega(\R^d,\FibD{})
    ,\partial+D)\stackrel{\lambda_D}{\longleftarrow}
    (\Dpoly{}(\R^d),\partial)
  \end{equation}
  of $L_\infty$-quasi-isomorphisms. Also, note that
  $D:=-\delta+\D$ follows from the natural choice $\conn=\D$. We obtain this map $\mathscr{U}^{A_{-1}}$ as follows.
  First we note that, by applying the map $\mathscr{K}$ from Theorem
  \ref{formalformality} fiberwise, we obtain the $L_\infty$-morphism
  \begin{equation*}
    \mathscr{U}
    \colon 
    (\Omega(\R^d;\FibT{}),\D)
    \longrightarrow 
    (\Omega(\R^d;\FibD{}),\partial+\D).
  \end{equation*}
  By considering the filtrations by exterior degree on both these
  algebras we construct spectral sequences which show that
  $\mathscr{U}$ is a quasi-isomorphism. Using this same filtration we
  may consider the MC element
  $A_{-1}\in\mathcal{F}^1\Omega(\R^d;\FibT{})$. Now note that
  $\Omega(\R^d;\FibT{})^{A_{-1}}$ is exactly
  $(\Omega(\R^d;\FibT{}),D)$ and
  $\Omega(\R^d;\FibD{})^{{A_{-1}}_\mathscr{U}}$ is exactly
  $(\Omega(\R^d;\FibD{}),D)$, since ${A_{-1}}_{\mathscr{U}}=A_{-1}$ by
  point (ii) of Theorem \ref{formalformality}. So we obtain the
  diagram \eqref{seqquisRd}.  Finally, to obtain the quasi-isomorphism
  \begin{equation*}
    \mathscr{U}
    \colon 
    \Tpoly{}(\R^d)\longrightarrow \Dpoly{}(\R^d)
  \end{equation*} 
  we need to invert the final arrow in the diagram \eqref{seqquisRd}.
  To do this we note that this arrow is actually an identification (by
  dgla-morphism) with the kernel of $D$ in exterior degree $0$.  Thus
  it can be inverted if we can guarantee that the map
  $\mathscr{U}^{A_{-1}}\circ\lambda_T$ maps $\Tpoly{}(\R^d)$ into this
  kernel. In \cite{Dolgushev2005} Dolgushev demonstrates a procedure to construct an $L_\infty$-morphism $\mathcal{V}$ homotopic to $\mathcal{U}^{A_{-1}}\circ\lambda_T$ that has this property.
\end{example}
%
%
%
The next step consists in the globalization, i.e. the generalization of the above result
to any manifold $M$. The only delicate part is now the twisting procedure. 
Dolgushev presents the twisting locally in a
way that is compatible on pairwise intersections of coordinate
charts.  This means that the global quasi-isomorphism is not
described, a priori, as a twist of another morphism.  In the following we
prove that the quasi-isomorphism of Fedosov resolutions is given by a twist of 
the fiberwise map $\mathscr{U} \colon \Omega(M;\FibT{}) \to \Omega(M;\FibD{})$, 
by showing that it induces a morphism of resolutions and using Prop.~\ref{prop:key}.
First, we need to find suitable resolutions of $\Omega(M;\FibT{})$ and $\Omega(M;\FibD{})$.
Let us fix a good cover $(U_i)_{i\in \mathcal{I}}$ of $M$ by coordinate neighborhoods. 
By abuse of notation we shall denote the set of $k$-tuples $(i_1,\ldots, i_k)$ in $\mathcal{I}$ 
such that $U_{i_1}\cap\ldots\cap U_{i_k}\neq\emptyset$ by $\mathcal{I}^k$. 
For $(i_1,\ldots, i_k)\in \mathcal{I}^k$ we shall denote $U_{i_1,\ldots, i_k}:=U_{i_1}\cap \ldots\cap U_{i_k}$.
As discussed in Example~\ref{ex:curvedlie}, $\left(\Omega(M;\FibT{}), \conn-\delta,  \Schouten{\argument,\argument}\right)$ 
and $\left(\Omega(M;\FibD{}), \partial+\conn-\delta, [\argument,\argument]_G\right)$, 
being curved Lie algebras, have the corresponding structure of $L_\infty$-algebras. Note that for each $a\in\mathcal{I}^k$ 
we obtain the curved Lie algebras $(\Omega(U_a;\FibT{}),Q^a)$ and $(\Omega(U_a;\FibD{}),P^a)$ by simply
restricting the structures on $\Omega(M;\FibT{})$ and $\Omega(M;\FibD{})$ respectively. 
In order to obtain an $L_\infty$-algebra structure on $\check{C}^i(\mathcal{I},\Omega(M;\FibT{})) = 
\prod_{a\in \mathcal{I}^i}\Omega(U_a;\FibT{})$ we need to introduce the notion of product of $L_\infty$-algebras. 
    
Let $\left\{(\mathfrak{L}_i,Q^i)\right\}_{i\in I}$ be a collection of $L_\infty$-algebras indexed over $I$. 
Set 
\begin{equation*}
  \prod_{i\in I}\mathfrak{L}_i:=\left\{f\in\Hom_{\mathsf{Set}}
  \left(I,\coprod_{i\in I}\mathfrak{L}_i\right)\mid
  f(i)\in\mathfrak{L}_i\right\}.
\end{equation*}
Note that $\prod_{i\in I}\mathfrak{L}_i$ is a vector space with 
$(f+g)(i)=f(i)+g(i)$, $0(i)=0$ and 
$(\lambda f)(i)=\lambda f(i)$ for all $i\in I$ and 
$\lambda$ scalar. Moreover, $\prod_{i\in I}\mathfrak{L}_i$ inherits a $\mathbb{Z}$-grading where
$f$ is homogeneous of degree $n$ if $f(i)$ is homogeneous of degree $n$ for all $i\in I$. 
Furthermore, we obtain the projections $\prod_{i\in I}\mathfrak{L}_i\rightarrow \mathfrak{L}_i$ 
as evaluation at $i\in I$. 
\begin{definition}[Product of $L_\infty$-algebras]
\label{def:prod}
  A graded vector space $\prod_{i\in I}\mathfrak{L}_i$ is called product of $L_\infty$-algebras if equipped 
  with the $L_\infty$-structure $Q$ given by the components 
   \begin{equation*}
    Q_0(1)
    =
    (i\mapsto Q_0^i(1))
    \quad\textnormal{  and  }\quad
    Q_k(f_1\vee\ldots\vee f_k)=(i\mapsto Q^i_k(f_1(i)\vee\ldots\vee f_k(i)))
  \end{equation*}
  for all $k\geq 1$.  We denote the product of $(\mathfrak{L}_i,Q^i)_{i\in
    I}$ by $\prod_{i\in I}(\mathfrak{L}_i,Q^i)$.
\end{definition}
\begin{example}[\v{C}ech complex]
	\label{Cechcomplex}
	By the above definition, we immediately obtain an $L_\infty$-algebra structure on 
    $\check{C}^i(\mathcal{I},\Omega(M;\FibT{}))=\prod_{a\in \mathcal{I}^i}\Omega(U_a;\FibT{})$ 
    for all $i\geq0$ and similarly for $\FibD{}$. 
    We denote these structures by $\mathcal{Q}^i$ and $\mathcal{P}^i$ respectively. 
\end{example}
As an immediate consequence of Definition~\ref{def:prod} 
we obtain the following lemma.
\begin{lemma}
  The limit of the discrete diagram of $L_\infty$-algebras $(\mathfrak{L}_i, Q^i)$ exists and is given by the 
  product $\prod_{i\in I}(\mathfrak{L}_i,Q^i)$. The relevant
  projection maps $\pr^i\colon\prod_{i\in I}(\mathfrak{L}_i,Q^i)\rightarrow (\mathfrak{L}_i,Q^i)$ 
  are given by the components $\pr^i_1(f) = f(i)$ and $\pr^i_k=0$ for all $k>1$.
\end{lemma}
Let $\{(\mathfrak{L}_i,Q^i)\}_{i\in I}$ and 
$\{(\mathfrak{K}_i,Q^i)\}_{i\in I}$ be two collections of
$L_\infty$-algebras and $F^i$ be morphisms from
$\mathfrak{L}_i$ to $\mathfrak{K}_i$. Then we obtain the morphisms 
\begin{equation*}
	F^j\circ \pr^j\colon \prod_{i\in I}(\mathfrak{L}_i, Q^i)\longrightarrow (\mathfrak{K}_j,P^j)
\end{equation*}
for all $j\in I$. By the universal property of the 
product we thus obtain the product morphism 
\begin{equation}
	\label{eq:prodmorph}
	F\colon \prod_{i\in I}(\mathfrak{L}_i,Q^i)\longrightarrow \prod_{i\in I}(\mathfrak{K}_i,P^i).
\end{equation}
\begin{lemma}
	\label{lem:pilist}
	Given a collection of elements $\pi_i\in\mathcal{F}^1\mathfrak{L}_i[1]^0$
    and the corresponding $\pi\in \mathcal{F}^1\prod_{i\in I}\mathfrak{L}_i[1]^0$ 
    given by $(i\mapsto \pi_i)$, we have the following properties:
    \begin{lemmalist}
  		\item The $L_\infty$-algebra $\left(\prod_{i\in I}\mathfrak{L}_i,Q^\pi\right)$ is naturally
           $L_\infty$-isomorphic to $\prod_{i\in I}(\mathfrak{L}_i,(Q^i)^{\pi_i})$ .
    	\item Given a collection of morphisms $F^i$ between $\mathfrak{L}^i$ and $\mathfrak{K}^i$, 
           the collection $(F^i)^{\pi_i}$ induces the $L_\infty$-morphism $F^\pi$ using the notation of \eqref{eq:prodmorph}. 
  		\item If all the elements $\pi_i$ are Maurer--Cartan, then $\pi$ is Maurer--Cartan.
  \end{lemmalist}
\end{lemma}
\begin{proof}
  The first claims follows easily from the formula 
  \begin{equation*}
      Q_k^\pi(\gamma_1,\ldots, \gamma_k)
      =
      \sum_{\ell=0}^\infty \frac{1}{\ell!}Q_{k+\ell}(\pi^\ell\vee\gamma_1\vee\ldots\vee\gamma_k)
  \end{equation*}
  for $Q^\pi$. The second claim follows from the similar expression 
  \begin{equation*}
      F_k^\pi(\gamma_1\vee\ldots\vee\gamma_k)
      =
      \sum_{\ell=0}^\infty\frac{1}{\ell!}F_{k+\ell}(\pi^\ell\vee\gamma_1\vee\ldots\vee\gamma_k)
  \end{equation*}
  for the twist of an $L_\infty$-morphism. Finally by the first point we have 
  $Q^\pi(1)(i)=(Q^i)^{\pi_i}(1)$ which shows the third claim. 
\end{proof}
Recall the $L_\infty$-structures on the components of the 
\v{C}ech complex $\check{C}^\bullet(\mathcal{J};\Omega(M;\FibT{}))$ from Example~\ref{Cechcomplex}. For $i\geq 0
$ and each $a\in \mathcal{J}^i$ the restriction map from 
$\Omega(M;\FibT{})$ to $\Omega(U_a;\FibT{})$ induces the 
$L_\infty$-morphism $R^a$ where $R^a_i=0$ for all $i>1$ and
where $R^a_1$ is the restriction map. By the universal 
property of the product these maps combine into the 
$L_\infty$-morphisms 
\begin{equation}
	\mathcal{R}^i\colon \Omega(M;\FibT{})\longrightarrow \check{C}^i(\mathcal{J};\Omega(M;\FibT{})).
\end{equation}
We equip the components $\check{C}^i(\mathcal{J};\Omega(M;\FibT{}))$ of the \v{C}ech complex with the structures 
$\varphi^i$ of $L_\infty$-modules over $\FormsT{}{M}{}$ 
induced from the maps $\mathcal{R}^i$  as in Example~\ref{morphmodule}.
In the case of $\FormsD{}{M}{}$ we denote the restriction maps by $\mathcal{A}^i$. Note that, since $\conn$ is a \emph{linear} connection and by 
Theorem~\ref{formalformality} \refitem{formalformality3}, the fiberwise application of
$\mathscr{K}$ yields the $L_\infty$-morphism $\mathscr{U}$ from $\Omega(M;\FibT{})$ to
$\Omega(M;\FibD{})$ with the $L_\infty$-structures with first Taylor coefficients 
$\conn-\delta$ and $\partial+\conn-\delta$ respectively. 
Thus we may equip the components 
$\check{C}^i(\mathcal{J};\Omega(M;\FibD{}))$ of the \v{C}ech
complex of $\FormsD{}{M}{}$ with the $\FormsT{}{M}{}$-module 
structures $\psi^i$ obtained through example \ref{morphmodule} by
the maps $\mathcal{A}^i\circ\mathscr{U}$. Of course we may also 
consider $\FormsD{}{M}{}$ itself as an $\FormsT{}{M}{}$ module by
applying example \ref{morphmodule} to the map $\mathscr{U}$.  
\begin{lemma}
\label{lem:formality1}
  The sequence 
  \begin{equation}0\rightarrow\Omega(M;\FibT{})\stackrel{\mathcal{R}^0}{\longrightarrow} 
  \check{C}^0(\mathcal{J};\Omega(M;\FibT{}))\stackrel{\partial_0}{\longrightarrow}\check{C}^1(\mathcal{J};\Omega(M;\FibT{}))\stackrel{\partial_1}{\longrightarrow}\ldots\end{equation}
  forms a resolution of $\Omega(M;\FibT{})$ and there is a similar resolution $\mathcal{A}^0$ of $\FibD{}$.
\end{lemma}
\begin{proof}
  The maps $\partial_i$ are given by the components
  $(\partial_i)_j=0$ for all $j>0$ and $(\partial_i)_0$ simply 
  the \v{C}ech differential. Thus it is clear that 
  $\partial_0\circ\mathcal{R}^0=0$ and
  $\partial_{i+1}\circ\partial_i=0$ for all $i\geq 0$ and we 
  only need to show that the maps $\partial_i$ and $\mathcal{R}^0$ indeed define 
  $L_\infty$-module morphisms. 
  
  Note that $\mathcal{R}^0$ is automatically a map of $L_\infty$-modules by Example~\ref{morphmodulemorph}, 
  but the same is not true of the $\partial_j$ since they do not preserve the $L_\infty$-algebra structure. Note that by
  definition of the maps $\partial_j$ we have that 
  $\partial_j(\gamma_1\vee\ldots\vee\gamma_n\otimes m) = \gamma_1\vee\ldots\vee\gamma_n\otimes \partial_jm$. 
  Thus the fact that the $\partial_j$ are maps of $L_\infty$-modules follows from the fact that 
  \begin{equation}\label{eq:partialphi}
  	\partial_j\varphi^j_k(\gamma_1\vee\ldots\vee\gamma_n\otimes m)
    =
    \varphi_k(\gamma_1\vee\ldots\vee\gamma_n\otimes \partial_jm).
  \end{equation}
  This is established by simply writing out the definitions of both sides. 
  
  \vspace{0.3cm}
  
  Similarly the map $\mathcal{A}^0$ yields a map of $L_\infty$-modules by applying 
  Example~\ref{morphmodulemorph} and Eq.~\eqref{eq:partialphi} holds with $\varphi$ 
  replaced by $\psi$. As with $\FormsT{}{M}{}$ we have $\partial_0\circ\mathcal{A}^0=0$ and 
  $\partial_{j+1}\circ\partial_j=0$ for all $j\geq 0$. 
\end{proof}
Replacing $M$ by $U_a$ with $a\in\mathcal{I}^k$ and $k\geq 0$ in the discussion of $\mathscr{U}$ preceding 
Lemma~\ref{lem:formality1} and inducing maps to the product (as done in the discussion after 
Lemma~\ref{lem:pilist}), we obtain the $L_\infty$-morphisms 
\begin{equation}
    \mathscr{U}^k
    \colon 
    \check{C}^k\left(\mathcal{I},\FormsT{}{M}{})\right)\longrightarrow \check{C}^k\left(\mathcal{I};\FormsD{}{M}{}\right)
\end{equation}
for all $k\geq 0$. Note that for each $k\geq 0$ this yields the commuting diagram 
\begin{equation}
    \begin{tikzcd}
      & \FormsT{}{M}{}
      \arrow[r, "\mathcal{R}^0"]\arrow[d, swap, "\mathscr{U}"] 
      & \check{C}^k\left(\mathcal{I};\FormsT{}{M}{}\right)
      \arrow[d,  "\mathscr{U}^k"]
      \\
      & \FormsD{}{M}{}\arrow[r, "\mathcal{A}^0"]
      & \check{C}^k\left(\mathcal{I};\FormsD{}{M}{}\right)
    \end{tikzcd}
\end{equation}
which allows us to realize the maps $\mathscr{U}^k$ as maps of $L_\infty$-modules over $\FormsT{}{M}{}$ by using 
Example~\ref{morphmodulemorph}. We denote these maps by $\mathscr{U}^k$ again since it should not cause any confusion. 
 As an immediate consequence we can prove the following
lemma.
\begin{lemma}
	\label{lem:morphresol}
	The maps $\mathscr{U}$ and $\mathscr{U}^\bullet$ form a morphism from the resolution $\mathcal{R}^0$ 
    to the resolution $\mathcal{A}^0$.
\end{lemma}
The next step consists in constructing a resolution adapted Maurer--Cartan element.
Let us consider the one-forms $A_{-1}$ and $A$ in $\Omega(M;\FibT{})$ introduced above
in order to construct the Fedosov differential as in \eqref{eq:fedosovD}.
%
\begin{lemma}
\label{lem:gMC}
	The element $\Gamma := A_{-1} + A\in \FormsT{1}{M}{0}$ is a Maurer--Cartan 
    element adapted to both $\mathcal{R}^0$ and $\mathcal{A}^0$. 
\end{lemma}
\begin{proof}
  First of all we note that the filtration appearing in the definition of a Maurer--Cartan 
  element is given here by the exterior degree. 
  Consider the Eq.~\eqref{eq:fedosovD} and note that the condition $D^2=0$ says that 
  \begin{equation*}
  	\conn^2 + \Schouten{\nabla\Gamma,\argument} + \frac{1}{2}\Schouten{\Schouten{\Gamma,\Gamma},\argument}
    =
    0,
  \end{equation*} 
  which (together with $\delta^{-1}A=0$) implies that
  \begin{equation*}
      R+\conn\Gamma +\frac{1}{2}\Schouten{\Gamma,\Gamma}
      =
      0,
  \end{equation*}
  where we have denoted the curvature of $\conn$ by $R$, i.e. $R\in\FormsT{2}{M}{}$ is 
  defined by $\conn^2\alpha=\Schouten{R,\alpha}$.  
  This shows that $\Gamma$ is an MC-element and it is left to show that it is 
  adapted to both $\mathcal{R}^0$ and $\mathcal{A}^0$. 

  Suppose we have that $\partial_j^\pi=\partial_j$ for all $j\geq 0$. Then $\pi$ would obviously be adapted to both $\mathcal{R}^0$ and 
  $\mathcal{A}^0$, since the twisted resolution would simply be
  the \v{C}ech complex of polyvector fields and 
  polydifferential operators respectively. These two sheaves 
  are fine, thus the corresponding \v{C}ech complexes on 
  a good cover are acyclic. Now simply note that indeed 
  $\partial_j^\pi=\partial_j$ for all $j\geq 0$ since the 
  Taylor coefficients of these maps vanish except in the lowest 
  order and we may use the formula from the proof of Lemma~\ref{lem:pilist}.
\end{proof}
Finally, using the techniques introduced above we can give another proof of
the formality theorem that we state below.
\begin{theorem}\label{Mformality}
	The dgla's $\Tpoly{}(M)$ and $\Dpoly{}(M)$ are quasi-isomorphic.
\end{theorem}
\begin{proof}
	We prove that the $L_\infty$-morphisms
    \begin{equation}
    	\label{seqquisM}
  		(\Tpoly{}(M),0)\stackrel{\lambda_T}{\longrightarrow}
  		(\Omega(M;\FibT{}),D)\stackrel{\mathscr{U}^\Gamma}{\longrightarrow}
  		(\Omega(M,\FibD{}),\partial+D)\stackrel{\lambda_D}{\longleftarrow}
  		(\Dpoly{}(M),\partial)
	\end{equation}
	are all quasi-isomorphisms.  Then, in order to obtain the
    quasi-isomorphism $\Tpoly{}(M)\longrightarrow \Dpoly{}(M)$
	we only need to invert the final arrow in  \eqref{seqquisM}.
    This is done in the same way as in Example \ref{formalityRd} (see
    \cite[section 4.2]{Dolgushev2005a}). 
    From Prop~\ref{lem:morphresol}, $\mathscr{U}$ is an $L_\infty$-morphism of resolutions. 
    Thus, using Lemma~\ref{lem:gMC}, Proposition~\ref{prop:key}, Lemma~\ref{lem:pilist} and 
    Example~\ref{morphmodulemorph} it is enough to show that 
    $\mathscr{U}^\Gamma$ is a quasi-isomorphism on the $U_a$ 
    for $a\in \mathcal{J}^k$ and $k\geq 1$. To show this we note first that 
    \begin{equation*}
    	\mathscr{U}
        \colon 
        \left(\FormsT{}{U_a}{}, \D, \Schouten{\argument,\argument}\right)
        \longrightarrow 
        \left(\FormsD{}{U_a}{}, \D, [\argument,\argument]_G\right)
    \end{equation*}
    is a well-defined $L_\infty$-quasi-isomorphism. On $U_a$ we have the decomposition 
    $\conn=\D+\Schouten{B_a,\argument}$ for some $B_a\in\FormsT{}{U_a}{}$. 
    By Theorem~\ref{formalformality} it follows that 
    \begin{equation*}
    	\mathscr{U}^{B_a+\Gamma}
        =
        \mathscr{U}^\Gamma\colon \left(\FormsT{}{U_a}{}, D, \Schouten{\argument,\argument}\right)
        \longrightarrow \left(\FormsD{}{U_a}{}, D, [\argument, \argument]_G\right).
    \end{equation*}
    Now since $\mathscr{U}$ was a quasi-isomorphism this proves that $\mathscr{U}^\Gamma$ 
    is a quasi-isomorphism, since $B_a+\Gamma$ is an MC-element (see \cite{Dolgushev2005}). 
\end{proof}
\begin{remark}[Formality for Lie algebroids]
	Formality for Lie algebroids has been proved in \cite{Calaque2005}
    and also uses the twisting procedure. We here remark that the techniques
    discussed above also apply to formality for Lie algebroids. Similarly,
    the authors conjecture that this observation immediately extend to the result
    presented in \cite{Chemla}.
\end{remark}

%
\subsubsection{Formality for Hochschild chains}

Formality for Hochschild chains has been conjectured in \cite{Tsygan99}
and proved by Dolgushev in \cite{Dolgushev2006} by using the globalization
techniques proposed in \cite{Dolgushev2005,Dolgushev2005a} and the local 
formality for Hochschild chains proved by Shoikhet in \cite{Shoikhet03}.
Here we briefly recall the Fedosov resolutions proved in \cite[section 4]{Dolgushev2006}. 
In the previous we were concerned with the analog $\Dpoly{}(M)$ of Hochschild cochains on $\Cinfty(M)$. 
Similarly we consider the analog $\Cpoly{}(M)$ of Hochschild chains on $M$ given by 
\begin{equation}
	 \Cpoly{-n}(M)= \Cinfty(M^{n+1}),
    \qquad
    \Cpoly{0} (M) = \Cinfty(M),
\end{equation}
where $M^{n+1}$ denotes the $n+1$-fold Cartesian product of $M$ with itself. 
The space $\Cpoly{}(M)$ can be naturally endowed with a structure of graded module over
the Lie algebra $\Dpoly{}(M)$ and we denote the corresponding action by $\rho$. The multiplication
$\mu$ in the algebra $\Cinfty(M)$ induces a differential $\mathrm{b}$ on $\Cpoly{}(M)$
by
\begin{equation*}
	\mathrm{b} := \rho_\mu 
    \colon 
    \Cpoly{n} (M) \longrightarrow \Cpoly{n+1}(M).
\end{equation*}
It is easy to see that $(\Cpoly{}(M), \mathrm{b})$ is a dg module over the dgla $\Dpoly{}(M)$
(see \cite[section 3]{Dolgushev2006}). 
Its cohomology is isomorphic, as a vector space, to the space $A^\bullet (M)$ of forms on $M$ with an inverted
grading, as proved by Teleman in \cite{Teleman98}. The dg $\Dpoly{}(M)$-module structure 
on $\Cpoly{}(M)$ induces a dg $\Tpoly{}(M)$-module structure on $A^\bullet (M)$, with action denoted 
by $\lambda$ (defined by the action of a polyvevtor field on exterior forms via the Lie derivative). 
The first step to find Fedosov resolutions for $\Cpoly{}(M)$ and $A^\bullet (M)$
consists in a local statement. Note that composing 
the quis $\mathscr{U}$ discussed in Example~\ref{formalityRd} with the action $\rho$
we obtain an $L_\infty$-module structure $\psi$ on $\Cpoly{}(\R^d))$ over 
the dgla $\Tpoly{}(\R^d)$.
\begin{theorem}[Shoikhet, \cite{Shoikhet03}]
   \label{thm:shoikhet}
 	There exists a quis
    \begin{equation*}
	  	\mathscr{S}
        \colon 
        (\Cpoly{}(\mathbb{R}^d), \mathrm{b}) \longrightarrow (A^\bullet (\mathbb{R}^d), 0 )
    \end{equation*}
    of $L_\infty$-modules over $\Tpoly{}(\mathbb{R}^d)$, with actions given by $\psi$ and $\lambda$ respectively
    and where $\mathscr{S}_0$ is given by Teleman's theorem
    and satisfying the same properties of 
    Theorem~\ref{formalformality}
\end{theorem} 
Let us denote the bundle of \emph{formal fiberwise Hochschild chains} whose fibers
are dg $\Dpoly{}(\mathbb{R}^d)$-modules by $\FibC{}$. Similarly, we consider the bundle $\mathcal{E}$ of 
\emph{formal fiberwise exterior forms}, i.e. exterior 
forms with values in the bundle of the formally completed symmetric algebra $\SM$  of $T^* M$.
Clearly $(\Omega(M;\mathcal{E}), 0)$ and $(\Omega(M;\FibC{}), \mathrm{b})$ are fiberwise dgla modules
over $\Omega(M;\FibT{})$ and $\Omega(M;\FibD{})$, respectively. We denote the fiberwise Lie derivative 
on $\Omega(M;\mathcal{E})$ and the fiberwise action of $\Omega(M;\FibD{})$ on $\Omega(M;\FibC{})$
again by $\lambda$ and $\rho$, resp. Also, the differential on $\Omega(M;\FibC{})$ can be written as 
$\mathrm{b} = \rho_\mu $ with $\mu \in \FibD{1}$.
Finally, in complete analogy with the above discussion,
one obtains the following statement.
\begin{theorem}[Fedosov Resolutions, Dolgushev \cite{Dolgushev2006}]
  \label{FedRes}
 There exist quasi-isomorphisms of dgla modules
  \begin{equation*}
    \lambda_A
    \colon 
    \left(A^\bullet(M), 0 \right)\longrightarrow			\left(\Omega(M;\mathcal{E}), D\right)
    \quad\text{and}\quad
    \lambda_C
    \colon 
    \left(\Cpoly{}(M), \rho_{\mu} \right)\longrightarrow \left(\FormsC{}{M}{},D+ \rho_\mu \right),
  \end{equation*} 
  where $D$ denotes the Fedosov differential.
\end{theorem}
Composing the fiberwise quis $\mathscr{U}\colon (\Omega(M;\FibT{}),0)\longrightarrow 
(\Omega(M,\FibD{}),\partial)$ with the fiberwise action $\rho$ of $\Omega(M,\FibD{})$
on $\Omega(M,\FibC{})$ we obtain an $L_\infty$-module structure on $\Omega(M,\FibC{})$ 
over $\Omega(M,\FibT{})$, also denoted by $\psi$. 
Moreover, from Theorem~\ref{thm:shoikhet} we obtain a fiberwise quis, denoted by $\mathscr{V}$ between the 
modules $(\Omega(M,\FibC{}), \mathrm{b}, \psi)$ and $(\Omega(M;\mathcal{E}), 0, \lambda)$. Thus 
we find that $\mathscr{V}$ is a morphism of the 
$L_\infty$-modules $(\Omega(M,\FibC{}),\mathrm{b})$ and $(\Omega(M;\mathcal{E}),0) $
over $(\Omega(M;\FibT{}),0, \Schouten{\argument, \argument})$. 
It is easy to observe that, in analogy with last section, 
the sequence 
\begin{equation}
	0\rightarrow
    \Omega(M;\FibC{})\stackrel{\mathcal{R}^0}{\longrightarrow} 
    \check{C}^0(\mathcal{J};\Omega(M;\FibC{}))\stackrel{\partial_0}{\longrightarrow}\check{C}^1(\mathcal{J};\Omega(M;\FibC{}))
    \stackrel{\partial_1}{\longrightarrow}\ldots
\end{equation}
forms a resolution of $\Omega(M;\FibC{})$. Similarly, there is a resolution of $\Omega(M;\mathcal{E}) $.
Using the same argument as for the morphism $\mathscr{U}$ we obtain:
\begin{lemma}
	 $\mathscr{V} \colon \Omega(M,\FibC{}) \longrightarrow \Omega(M;\mathcal{E})$ is a
     morphism of resolutions.
\end{lemma}
Finally, as an immediate consequence of the above lemma we can prove formality 
for Hochschild chains as follows.
\begin{theorem}
	The dg modules $(\Tpoly{}(M), A^\bullet (M))$ and $(\Dpoly{}(M), \Cpoly{}(M))$
    are quasi-isomorphic.
\end{theorem}
\begin{proof}
 	Twisting the resolution morphism $\mathscr{V}$ by $\Gamma := A_{-1} + A\in \FormsT{1}{M}{0}$
    and using Prop.~\ref{prop:key} 
    we obtain the $L_\infty$-quasi-isomorphism 
    $\mathscr{V}^\Gamma \colon (\Omega(M,\FibC{}), D + \rho_\mu) \longrightarrow (\Omega(M;\mathcal{E}),D)$ as in theorem \ref{Mformality}.
    In fact, given the $L_\infty$-quasi-isomorphism $\mathscr{U}^\Gamma$ it is not hard to show
    that the dgla module structures on $\Omega(M;\mathcal{E})$ and $\Omega(M,\FibC{})$ over $(\Omega(M;\FibT{}),D)$
    and $(\Omega(M;\FibD{}),D + \partial)$, resp, obtained by twisting via $\Gamma$ coincide with those
    defined by the fiberwise structures $\lambda$ and $\rho$.
   This concludes the proof.
\end{proof}

\begin{remark}[Formality for chains in the Lie algebroid setting]
	The same techniques can be used to prove formality 
    for Hochschild chains in the Lie algebroid setting,
    whose original proof can be found in \cite{Calaque2006}.
\end{remark}

\end{document}